\documentclass{article}
\usepackage[utf8]{inputenc}
\usepackage{amssymb}
\usepackage{amsfonts}
\usepackage[english]{babel}
\usepackage{amsmath,amscd}
\usepackage{amsthm}
\usepackage{a4wide}
\usepackage{color}
\usepackage{mathrsfs}
\usepackage{array}
\usepackage{tabularx}

\title{Derivations of Quantum and Involution Generalized Weyl Algebras}
\author{Andrew P. Kitchin}
\date{}

\newtheorem{theorem}{Theorem}[section]
\newtheorem{proposition}[theorem]{Proposition}
\newtheorem{lemma}[theorem]{Lemma}
\newtheorem{corollary}[theorem]{Corollary}

\theoremstyle{definition}
\newtheorem*{remark}{Remark}
\newtheorem*{acknowledgements}{Acknowledgements}

\theoremstyle{definition}

\newcommand{\Dmath}{\mathop{}\!\mathrm{D}}

\begin{document}
\maketitle
\begin{abstract} We classify the derivations of degree-one generalized Weyl algebras over a univariate Laurent polynomial ring. In particular, our results cover the Weyl-Hayashi algebra, a quantization of the first Weyl algebra arising as a primitive factor algebra of $U_q^+(\mathfrak{so}_5)$, and a family of algebras which localize to the group algebra of the infinite group with generators $x$ and $y$, subject to the relation $xy = y^{-1}x.$
\end{abstract}

\section{Introduction}
Let $\Bbbk$ be a field. For a $\Bbbk$-algebra $R$, a ($\Bbbk$-algebra) \textbf{derivation} is a $\Bbbk$-linear map $\textrm{D}:R\rightarrow R$ satisfying Leibniz's law
$$\textrm{D}(ab)=\textrm{D}(a)b+a\textrm{D}(b)$$
for all $a,b\in R.$ A derivation is called \textbf{inner} if for $b\in R$ we have $\textrm{D}(a)=ba-ab$ for all $a\in R$. We follow convention and use the term adjoint and notation $\mathrm{ad}_b(a)=ba-ab,$ when referring to inner derivations. A derivation is called \textbf{outer} if it cannot be expressed as the sum of inner derivations. The study of derivations in both commutative and noncommutative settings has been ubiquitous due to their structural importance and link to automorphisms, for example see \cite{Kharchenko}. 

This article is devoted to classifying the derivations of a family of generalized Weyl algebras over a univariate Laurent polynomial ring. We first recall the definition of a generalized Weyl algebra, as introduced in \cite{Bavula2}, highlighting the canonical forms over polynomial and Laurent polynomial rings, as well as notable examples for which the derivations have previously been studied. 

For a $\Bbbk$-algebra $R$, a ($\Bbbk$-algebra) automorphism $\sigma$ of $R$, and a central element of $R$, say $a$, the degree-one \textbf{generalized Weyl algebra} $R(\sigma, a)$ is the algebra extension of $R$ by the two indeterminates $x$ and $y$ subject to the relations
$$yx=a,~~xy=\sigma(a),~~xr=\sigma(r)x,~~\textrm{and}~~ yr=\sigma^{-1}(r)y~~\textrm{for all}~r\in R.$$
Consider the case where $R=\Bbbk[h]$. Every automorphism, say $\sigma$, of $\Bbbk[h]$ has the action $\sigma(h)=q h+\beta$, where $q\in\Bbbk^*$ and $\beta\in\Bbbk$. The following classification of the canonical forms of $\Bbbk[h](\sigma,a)$ appeared in \cite{Richard1}.
\begin{proposition}\label{prop_richard_solotar}(Richard, Solotar \cite[Proposition 2.1.1]{Richard1})
The generalized Weyl algebra $\Bbbk[h](\sigma,a)$ is isomorphic to one, and only one, of the following
\begin{enumerate}
\item $\Bbbk[h](\mathrm{id},a)$,
\item $\Bbbk[h](\sigma_{cl},a)$ where $\sigma_{cl}(h)=h+1$,
\item $\Bbbk[h](\sigma_q,a)$ where $\sigma_q(h)=qh$ and $q\neq 1$.
\end{enumerate}
\end{proposition}
\noindent
The notation $\sigma_{cl}$ and $\sigma_{q}$ reflects the common use of the words \textbf{classical} and \textbf{quantum} to describe families 2 and 3 respectively. 

In \cite{Dixmier} Dixmier proved that every derivation of the first Weyl algebra is inner. The Weyl algebra is isomorphic to the classical generalized Weyl algebra $\Bbbk[h](\sigma_{cl}, h)$. In \cite{Farinati1} it was shown that the first Hochschild cohomology group of $\Bbbk[h](\sigma_{cl},a)$ is zero. Recalling that the first Hochschild cohomology group of an algebra, say $A$, can be interpreted as the Lie algebra of derivations of $A$ modulo the inner derivations, one can conclude that Dixmier's result on the derivations of the Weyl algebra is true for $\Bbbk[h](\sigma_{cl},a)$, irrespective of the choice of $a$. 

The derivations of the quantum plane, the $\Bbbk$-algebra generated by the symbols $x$ and $y$ subject to the relations $xy = qyx$, were classified in \cite{Alev2}. The quantum plane can be realized as the quantum generalized weyl algebra $\Bbbk[h](\sigma_{q},h)$. 

In this article we consider the family of generalized Weyl algebras where $R=\Bbbk[h^{\pm 1}]$. Every automorphism, say $\sigma$, of $\Bbbk[h^{\pm 1}]$ has the action $\sigma(h)=q h^{\pm 1}$, where $q\in\Bbbk^*$. We continue refer to the family $\Bbbk[h^{\pm 1}](\sigma,a)$, where $\sigma(h)=qh$, as quantum and refer to the family where $\sigma(h)=qh^{-1}$ as \textbf{involution} generalized Weyl algebras. We will prove the following Laurent polynomial analogue of Proposition \ref{prop_richard_solotar}.
\begin{proposition}\label{prop_canonical_forms}
The generalized Weyl algebra $\Bbbk[h^{\pm 1}](\sigma, a)$ is isomorphic to one, and only one, of the following
\begin{enumerate}
\item $\Bbbk[h^{\pm 1}](\mathrm{id},a)$,
\item $\Bbbk[h^{\pm 1}](\sigma_q,a)$ where $\sigma_q(h)=qh$ and $q\neq 1$,
\item $\Bbbk[h^{\pm 1}](\sigma_{inv},a)$ where $\sigma_{inv}(h)=qh^{-1}$.
\end{enumerate}
\end{proposition}
\noindent
For reasons of presentation, we will routinely switch between the notations $a$ and $a(h)$ to denote the defining polynomial of $\Bbbk[h^{\pm 1}](\sigma,a)$.

In \cite{Launois}, primitive factor algebras of Gelfand-Kirillov dimension 2 of the positive part of the quantized enveloping algebra $U_q(\mathfrak{so}_5)$ were classified. Among those are the algebras $\mathscr{A}_{\alpha,q}$ with $\alpha\in\Bbbk^*$, where $\mathscr{A}_{\alpha,q}$ is the associative algebra in three variables $e_1, e_2,$ and $e_3$ subject to the relations
\begin{align*}
& e_1 e_3=q^{-1}e_3 e_1,\\
& e_2 e_3=q e_3 e_2+\alpha,\\
& e_2e_1=q^{-1} e_1 e_2 -q^{-1}e_3,\\
& e_3^2+(q^2-1) e_3 e_1 e_2 +\alpha q(q+1) e_1=0.
\end{align*}
By setting $q=1$ and $\alpha=1$, we get an algebra isomorphic to the first Weyl algebra. In \cite{Launois} these algebras are denoted $A_{\alpha,0}$, and for simplicity, we replace $q^2$ with $q$. When $q\neq 1$ the algebra $\mathscr{A}_{\alpha,q}$ is isomorphic to $\Bbbk[h^{\pm 1}](\sigma_q,h-1).$ 

Let $\mathscr{H}_{q}^t$ denote the associative algebra with generators $\Omega, \Omega^{-1}, \Psi$, and $\Psi^{\dag}$ subject to the relations 
\begin{align}
& \Omega\Omega^{-1}=\Omega^{-1}\Omega=1,\nonumber\\
& \Psi \Omega=q\Omega\Psi,\nonumber\\ &\Psi^{\dag} \Omega=q^{-1} \Omega \Psi^{\dag},\nonumber\\
&\Psi \Psi^{\dag}=\frac{q^t \Omega^t-q^{-t}\Omega^{-t}}{q^t-q^{-t}},\nonumber\\
&\Psi^{\dag}\Psi=\frac{ \Omega^t-\Omega^{-t}}{q^t-q^{-t}}.\nonumber
\end{align}
By setting $t=1$, we retrieve the Weyl Hayashi algebra $\mathscr{H}_{q}^1$ which was studied in \cite{Alev} and \cite{Kirkman}. When $t=2$ we get the original algebras introduced by Hayashi in \cite{Hayashi}. In \cite{Kirkman} it was shown that the algebras $\mathscr{H}_{q}^1$ arise as factor algebras of a $q$-analogue of the universal enveloping algebra of the Heisenberg Lie algebra. It was also shown in \cite{Launois} that $\mathscr{H}_{q}^1$ appears as factor algebras of the positive part of the quantized enveloping algebra $U_q(\mathfrak{so}_5)$. The algebras $\mathscr{H}_{q}^1$ and $\mathscr{H}_{q}^2$ are isomorphic to $\Bbbk[h^{\pm 1}](\sigma_q,h^2-1)$ and $\Bbbk[h^{\pm 1}](\sigma_q,h^4-1)$ respectively.   

For the quantum torus, $\Bbbk[h^{\pm 1}](\sigma_{q},h)$, an important localisation of the of the quantum plane, the derivations were classified as part a more general family of twisted group algebras \cite{Osborn}. We leverage this classification in Section \ref{DQGWA} to prove the following.
\begin{proposition}\label{Prop_main_1}
Let $q$ be a non-root of unity and $a(h)$ not a monomial.
\begin{enumerate}
    \item Every derivation of $\Bbbk[h^{\pm 1}](\sigma_q,a)$ can be uniquely written as
$\mathrm{ad}_t+\delta_{\alpha},$
where $t\in\Bbbk[h^{\pm 1}](\sigma_q,a)$ and $\delta_{\alpha}$ has action $$\delta_{\alpha}(h)=0,~~\delta_{\alpha}(x)=\alpha x,~~\mathrm{and}~~\delta_{\alpha}(y)=-\alpha y,$$
where $\alpha \in \Bbbk^*$.

\item The first Hochschild cohomology $\mathrm{HH}^1(\Bbbk[h^{\pm 1}](\sigma_q,a))=\Bbbk$. 
\end{enumerate}
\end{proposition} 

We dedicate Section \ref{IGWA1} and Section \ref{IGWA2} to studying involution generalized Weyl algebras. In Section \ref{IGWA1} we classify the units and center, as well as solve the isomorphism problem for these algebras. Using these results we prove Proposition \ref{prop_canonical_forms}. In Section \ref{IGWA2}, we arrive at the following classification of the derivations for this family.
\begin{proposition}\label{Prop_main_2}
Assume $\Bbbk$ does not have characteristic $2$ and $q$ has a second root in $\Bbbk$.
\begin{enumerate} 
    \item Every derivation of $\Bbbk[h^{\pm 1}](\sigma_{inv},a)$ can be uniquely written as $\mathrm{ad}_t+\delta_{z_1,z_2},$ where $t\in\Bbbk[h^{\pm 1}](\sigma_q,a)$ and $\delta_{z_1,z_2}$ has action
$$\delta_{z_1,z_2}(h)=(h^2-1)z_1,~~\delta_{z_1,z_2}(x)=z_2 x,~~\mathrm{and}~~\delta_{z_1,z_2}(y)=-z_2 y,$$
where $z_1,z_2 \in Z(\Bbbk[h^{\pm 1}](\sigma_{inv},a))$.
\item When $a(h)$ is not a monomial $z_1=0$. 
\end{enumerate}
\end{proposition} 
The prototypical example of an involution generalized Weyl algebra is the group algebra of the infinite group with generators $x$ and $y$ subject to the relation $xy = y^{-1}x.$ Namely, $\Bbbk[h^{\pm 1}](\sigma_{inv},1)$ is isomorphic to skew-Laurent ring $\Bbbk[v^{\pm 1}][u^{\pm 1};\sigma]$, where $\sigma(v)=v^{-1}$.
\\

\begin{acknowledgements}
The author thanks Professor St\'ephane Launois for guidance with the early stages of this work.  
\end{acknowledgements}

\section{Derivations of Quantum Generalized Weyl Algebras}\label{DQGWA}
Throughout this section we make the assumption that $q\in \Bbbk^*$ is not a root of unity.

\subsection{Derivations of the Quantum Torus}
We recall an important result of Osborn and Passman regarding the derivations of the quantum torus $\mathcal{T}_q$, the $\Bbbk$-algebra generated by the symbols $x, x^{-1}, y,$ and $y^{-1}$ subject to the relations $$xy=qyx~~\mathrm{and}~~xx^{-1}=yy^{-1}=1.$$ Originally the following result was proved in a more general setting, namely for twisted group algebras.
\begin{theorem}(Passman, Osbourn \cite[Corollary 2.3]{Osborn})\label{prop_passman_osbourn}
Every derivation of $\mathcal{T}_q$ can be uniquely written as $\mathrm{ad}_t+\delta_{\alpha,\beta}$ where $t\in\mathcal{T}_q$ and $\delta_{\alpha,\beta}$ has action
$$\delta_{\alpha,\beta}(u)=\alpha u~~\mathrm{and}~~\delta_{\alpha,\beta}(v)= \beta v$$
where $\alpha,\beta \in\Bbbk$.
\end{theorem}
\noindent
Theorem \ref{prop_passman_osbourn} has been useful in computing derivations in a number of settings, see for example \cite{Launois1} or \cite{Tang}. Likewise, we will use this to explicitly classify the derivations of $\Bbbk[h^{\pm 1}](\sigma_{q},a)$. 

Let $\Sigma$ denote the multiplicative system of $\Bbbk[h^{\pm 1}](\sigma_q,a)$ generated by $x$. By direct calculation we can verify that $\Sigma$ satisfies the Ore condition. Thus, we have the following localization of $\Bbbk[h^{\pm 1}](\sigma_q,a)$.
\begin{lemma}\label{loc12}
The localization $\Bbbk[h^{\pm 1}](\sigma_q,a)\Sigma^{-1}$ is isomorphic to the quantum torus $\mathcal{T}_q$.
\end{lemma}

\subsection{Derivations of $\Bbbk[h^{\pm 1}](\sigma_q,a)$}
In this section we will determine the form of any derivation of $\Bbbk[h^{\pm 1}](\sigma_q,a)$ when $a$ is not a monomial. When $a$ is a monomial $\Bbbk[h^{\pm 1}](\sigma_q,a)$ is isomorphic to the quantum torus $\mathcal{T}_q$ and hence we exclude it from our calculations. 

\begin{proposition}\label{K1}
Every derivation of $\Bbbk[h^{\pm 1}](\sigma_q,a)$ is of the form $\mathrm{ad}_t+\delta_{\alpha}$, where $t\in \Bbbk[h^{\pm 1}](\sigma_q,a)$, and $\delta_{\alpha}$ is the derivation of $\Bbbk[h^{\pm 1}](\sigma_q,a)$ defined on the generators by $$\delta_{\alpha}(h)=0,~~\delta_{\alpha}(x)=\alpha x,~~\mathrm{and}~~\delta_{\alpha}(y)=-\alpha y,$$ where $\alpha \in \Bbbk$. 
\end{proposition} 
\begin{proof}
Let $\textrm{D}$ be a derivation of $\Bbbk[h^{\pm 1}](\sigma_q,a)$. By Lemma \ref{loc12} we have that $\Bbbk[h^{\pm 1}](\sigma_q,a)\Sigma^{-1}\cong \mathcal{T}_q$ and therefore the derivation $\textrm{D}$ extends uniquely to a derivation of $\mathcal{T}_q$. By Theorem \ref{prop_passman_osbourn} we have that $\textrm{D}$ can be written as $\textrm{D}=\mathrm{ad}_t+\delta_{\alpha,\beta}$ where $t\in\mathcal{T}_q$, $\delta_{\alpha,\beta}(x)=\alpha x,$ and $\delta_{\alpha,\beta}(h)=\beta h$, for $\alpha,\beta\in \Bbbk$. For simplicity, we set $\delta_{\alpha,\beta}:=\delta$ for the duration of the proof.  

Using the decomposition $t=t_+ + t_-$, where $$t_+=\displaystyle\sum_{i\geq 0} c_i(h) x^i~~\mathrm{and}~~t_-=\displaystyle\sum_{i> 0} c_{-i}(h) x^{-i}$$
for $c_i(h)\in\Bbbk[h^{\pm 1}]$, we write
$$\textrm{D}=\mathrm{ad}_{t_+}+\mathrm{ad}_{t_-}+\delta.$$ 
Applying $\textrm{D}$ to $h^j$, for $j\in\mathbb{Z}_+$, gives us $$\mathrm{ad}_{t_+}(h^j)+\mathrm{ad}_{t_-}(h^j)+\delta(h^j).$$ 
By considering the form of $\delta$ and $\mathrm{ad}_{t+}$, we see that $\delta(h^j)$ and $\mathrm{ad}_{t_+}(h^j)$ are both elements of $\Bbbk[h^{\pm 1}](\sigma_q,a)$. Therefore, by the assumption that $\textrm{D}$ is a derivation of $\Bbbk[h^{\pm 1}](\sigma_q,a)$, we have that $\mathrm{ad}_{t_-}(h^j)\in \Bbbk[h^{\pm 1}](\sigma_q,a)$ for all $j>0$. 

We will now show that $c_i(h)x^{-i}\in \Bbbk[h^{\pm 1}](\sigma_q,a)$ for all $i>0$, proving that $t\in \Bbbk[h^{\pm 1}](\sigma_q,a)$. By direct calculation, we find that 
\begin{equation}\label{jifdskkkkrgij}
\mathrm{ad}_{t_-}(h^j)= \displaystyle\sum_{i> 0}^n c_{-i}(h) x^{-i}(1-q^{ij}) h^j.
\end{equation} 
After setting $e_j=-\mathrm{ad}_{t_-}(h^j)h^{-j}$ for $1<j<n$, Equation (\ref{jifdskkkkrgij}) can be written in the following way:
\[ \left( \begin{array}{ccccc}
(q-1) & (q^2-1) & (q^3-1) & \ldots & (q^{n}-1) \\
(q^{2}-1)& (q^{4}-1) &(q^{6}-1) & \ldots & (q^{2n}-1) \\
\vdots & \ddots & \ddots & \ddots & \vdots \\
\vdots & \ddots & \ddots & \ddots & \vdots \\
(q^{n}-1) & (q^{2n}-1) & (q^{3n}-1) & \ldots & (q^{n^2}-1) \\
\end{array} \right) 
\left( \begin{array}{c} c_{-1}(h)x^{-1}\\
c_{-2}(h)x^{-2}\\
\vdots\\
\vdots\\
c_{-n}(h)x^{-n}
 \end{array} \right)=\left( \begin{array}{c} e_1\\
e_2\\
\vdots\\
\vdots\\
e_n
 \end{array} \right) \]
 
\noindent 
where $e_i\in \Bbbk[h^{\pm 1}](\sigma_q,a)$ for $1\leq i \leq n$. To show $c_i(h)x^{-i}\in \Bbbk[h^{\pm 1}](\sigma_q,a)$ for all $i>0$, it suffices to show that the matrix  
\[ \textbf{M}:=\left( \begin{array}{ccccc}
(q-1) & (q^2-1) & (q^3-1) & \ldots & (q^{n}-1) \\
(q^{2}-1)& (q^{4}-1) &(q^{6}-1) & \ldots & (q^{2n}-1) \\
\vdots & \ddots & \ddots & \ddots & \vdots \\
\vdots & \ddots & \ddots & \ddots & \vdots \\
(q^{n}-1) & (q^{2n}-1) & (q^{3n}-1) & \ldots & (q^{n^2}-1) \\
\end{array} \right) \]

\noindent
is invertible. We proceed with this strategy. Let $\mathbf{x}:=(x_1,\ldots,x_n)\in(\Bbbk)^n$ and consider the polynomial 
$$g(r)=x_1 (r-1)+x_2(r^2-1)+\ldots+x_n( r^n-1).$$ 
Observe that $g(q^i)$ gives the $i^{\mathrm{th}}$ entry in the product of $\textbf{M}\cdot \textbf{x}^{\top}$. Assuming $\textbf{M}\cdot \textbf{x}^{\top}=0$ implies that $q^i$ is a root of $g$ for $1\leq i \leq n$. By examination, $g(1)=0,$ and since $q$ is not a root of unity, we have that there are at least $n+1$ distinct roots of $g$, contradicting the assumption $\textbf{M}\cdot \textbf{x}^{\top}=0$. Since we have shown that the null space of $\textbf{M}$ is trivial, $\textbf{M}$ is invertible.       

We now determine the action of $\delta$ on $y$. Since we have shown that $t\in \Bbbk[h^{\pm 1}](\sigma_q,a)$ we have, since $\textrm{D}$ is a derivation, $\delta$ is a derivation of $\Bbbk[h^{\pm 1}](\sigma_q,a)$. Thus, $\delta(y)\in \Bbbk[h^{\pm 1}](\sigma_q,a)$. Applying $\delta$ to the relation $yh=q^{-1}hy$ gives us $$\delta(y)h+y\beta h =q^{-1} \beta h y +q^{-1}h \delta(y)$$ which reduces to $\delta(y) h=q^{-1}h \delta(y)$. An application of Lemma 5.1 (iii) of \cite{Bavula1} gives us that $\delta(y)=p(h)y$ for $p(h)\in\Bbbk[h^{\pm1}]$. Applying $\delta$ to the relation $yx=a(h)$ gives us 
$$p(h)yx+\alpha yx=\delta(a(h)).$$
Rewriting $a(h)=\displaystyle\sum_{i\in\mathbb{Z}}a_i h^i$ allows us to see that 
$(p(h)+\alpha)\displaystyle\sum_{i\in\mathbb{Z}}a_i h^i=\displaystyle\sum_{i\in\mathbb{Z}}a_i \delta(h^i)$
which implies 
\begin{equation}\label{L3}
(p(h)+\alpha)\displaystyle\sum_{i\in\mathbb{Z}}a_i h^i=\displaystyle\sum_{i\in\mathbb{Z}}i \beta a_i h^i.
\end{equation}
By comparing coefficients of $h^i$ in Equation $(\ref{L3})$, we find that $\beta=0$ and $p(h)=-\alpha$.
\end{proof}

\begin{remark}
Since any derivation of $\Bbbk[h^{\pm 1}](\sigma_q,a)$ uniquely extends to a derivation of $\mathcal{T}_q$, Proposition \ref{Prop_main_1} follows directly from Proposition \ref{prop_passman_osbourn}.
\end{remark}

\section{Involution Generalized Weyl Algebras}\label{IGWA1}
To begin, we note that when studying $\Bbbk[h^{\pm 1}](\sigma_{inv},a)$ it is evident that when $a(h)$ is a monomial, say $a(h)=a_1h^i$ for $a_1\in\Bbbk$ and $i\in\mathbb{Z}$, the choice of $q$ is unimportant assuming, as we will from this point, $q$ has a second root in $\Bbbk$. We see this by considering the isomorphism $\Gamma$, between $\Bbbk[h^{\pm 1}](\sigma_{inv}',a_1 h^i)$ and $\Bbbk[h^{\pm 1}](\sigma_{inv},a_1 h^i)$, where $\sigma_{inv}'(h)=qh^{-1}$ and $\sigma_{inv}(h)=h^{-1}$, defined on the generators of $\Bbbk[h^{\pm 1}](\sigma_{inv}',a_1 h^i)$ by
$$\Gamma(\bar{h})=q^{\frac{1}{2}}h,~~\Gamma(\bar{x})=q^{\frac{i}{2}}x,~~\mathrm{and}~~\Gamma(\bar{y})=y.$$
We note that we have marked the generators of $\Bbbk[h^{\pm 1}](\sigma_{inv}',a)$ with a bar merely to distinguish them from the generators of $\Bbbk[h^{\pm 1}](\sigma_{inv},a)$. 

\begin{remark}
In the sequel, when considering $\Bbbk[h^{\pm 1}](\sigma_{inv},a)$ where $a$ is a monomial, we need only study the case where $\sigma_{inv}(h)=h^{-1}$. We will be somewhat cavalier with the notation $\sigma_{inv}$ using it to mean $\sigma_{inv}(h)=h^{-1}$ when $a$ is a monomial and $\sigma_{inv}(h)=qh^{-1}$, for $q\in\Bbbk^*$, when $a$ is not a monomial. 
\end{remark}

\subsection{Grading}\label{grading1}
The algebra $\Bbbk[h^{\pm 1}](\sigma, a)$ is $\mathbb{Z}$-graded for any choice of $\sigma$, see for example \cite{Bavula2}. Picking $\mathrm{deg}(h)=0, \mathrm{deg}(x)=1,$ and $\mathrm{deg}(y)=-1$, we have that
$$\Bbbk[h^{\pm 1}](\sigma, a)= \bigoplus_{n\in\mathbb{Z}}A_n,$$
where $A_n =\Bbbk[h^{\pm 1}]x^n$ if $n\geq 0$ and $A_n =\Bbbk[h^{\pm 1}]y^{-n}$ if $n<0$. Thus, any element in $\Bbbk[h^{\pm 1}](\sigma, a)$ can be expressed as 
$\displaystyle\sum_{i=0}p_i(h)x^i+\displaystyle\sum_{i=1}p_i'(h)y^i$, where $p_i(h), p_i'(h)\in \Bbbk[h^{\pm 1}]$.

\subsection{Center and Units}\label{center_and_units}
We begin by classifying the center and units for the algebra $\Bbbk[h^{\pm 1}](\sigma_{inv},a)$ when $a$ is a monomial. First we show that in this case $\Bbbk[h^{\pm 1}](\sigma_{inv},a)$ is a skew-Laurent ring. Note, since $a$ is a monomial we only consider the case where $\sigma_{inv}(h)=h^{-1}$ (see the remarks the begin Section \ref{IGWA1}).   
\begin{lemma}\label{iso1c}
Let $a\in\Bbbk[h^{\pm 1}]$ be a monomial. Then $\Bbbk[h^{\pm 1}](\sigma_{inv},a)$ is isomorphic to the skew Laurent extension $\Bbbk[v^{\pm 1}][u^{\pm 1}; \sigma]$, where $\sigma(v)=v^{-1}$. 
\end{lemma}
\begin{proof}
The map $\Gamma$ defined on the generators of $\Bbbk[h^{\pm 1}](\sigma_{inv},a)$ as follows:
$$\Gamma(h)=v,~~\Gamma(x)=u,~~\mathrm{and}~~\Gamma(y)=\alpha v^i u^{-1}$$
where $\alpha\in\Bbbk^*$ and $i\in\mathbb{Z}$,
provides the desired isomorphism.
\end{proof}
\begin{remark}
Since Lemma \ref{iso1c} gives us that $\Bbbk[h^{\pm 1}](\sigma_{inv},a)\cong \Bbbk[h^{\pm 1}](\sigma_{inv},1)$ when $a$ is a monomial, it is enough to consider $\Bbbk[h^{\pm 1}](\sigma_{inv},1)$.
\end{remark}

\begin{corollary}\label{Units1265}
Every unit of $\Bbbk[h^{\pm 1}](\sigma_{inv},1)$ is of the form $\alpha h^i x^j$ for $\alpha\in\Bbbk^*$ and $i,j\in\mathbb{Z}$. 
\end{corollary}
\begin{proof}
This follows by considering the isomorphism stated in Lemma \ref{iso1c}. 
\end{proof}

\begin{lemma}\label{center3423}
The center of $\Bbbk[h^{\pm 1}](\sigma_{inv},1)$ is isomorphic to the polynomial ring $\Bbbk[h+h^{-1},x^{\pm 2}]$.
\end{lemma}
\begin{proof}
Since $yx=1$ we replace $y$ by $x^{-1}$. Suppose $p$ is central. \`{A} la Section \ref{grading1}, we can express $p$ as the finite sum $p=\displaystyle\sum_{i\in\mathbb{Z}} p_i(h) x^i$ where $p_i(h)\in\Bbbk[h^{\pm{1}}]$. By assumption we have
$$ \displaystyle\sum_{i\in\mathbb{Z}} p_i(h) x^i x=x \displaystyle\sum_{i\in\mathbb{Z}} p_i(h) x^i$$
which implies  
$$ \displaystyle\sum_{i\in\mathbb{Z}} p_i(h) x^i x= \displaystyle\sum_{i\in\mathbb{Z}} p_i(h^{-1}) x^i x.$$ By comparing coefficients we find that $p_i(h)=p_i(h^{-1})$ for all $i\in\mathbb{Z}$. Similarly we have that
$$\displaystyle\sum_{i\in\mathbb{Z}} p_i(h) x^i h=h \displaystyle\sum_{i\in\mathbb{Z}} p_i(h) x^i$$ implying that $$\displaystyle\sum_{i\in\mathbb{Z}} h^{-1}p_{2i-1}(h) x^{2i-1} = \displaystyle\sum_{i\in\mathbb{Z}} h p_{2i-1}(h) x^{2i-1}$$ 
which yields that $p_{2i-1}(h)=0$ for $i\in \mathbb{Z}$. Thus, we have shown $p=\displaystyle\sum_{i\in\mathbb{Z}} p_i(h) x^{2i}$, where $p_i(h)=p_i(h^{-1})$. 

By induction we can show that $h+h^{-1}$ generates every polynomial of the form $p(h)=p(h^{-1})$ and we have that $h+h^{-1}$ and $x^2$ are algebraically independent.
\end{proof}

We now focus our attention on the case where $a$ is not a monomial and thus, we again consider $\sigma(h)=q h^{-1}$, where $q\in\Bbbk^*$. 
\begin{lemma}\label{fsddf5367}
For $a$ not a monomial, the center of $\Bbbk[h^{\pm 1}](\sigma_{inv},a)$ is isomorphic to the generalized Weyl algebra $\Bbbk[h+qh^{-1}](\mathrm{id},a(h)a(qh^{-1}))$.
\end{lemma}
\begin{proof}
Suppose $u$ is central. Let $u=\displaystyle\sum_{i\geq 0} p_i(h) x^i+\displaystyle\sum_{i>0} p_i'(h) y^i$ where $p_i(h),p_i'(h)\in\Bbbk[h^{\pm 1}]$. By assumption we have that
\begin{equation}
h\left[\displaystyle\sum_{i\geq 0} p_i(h) x^i+\displaystyle\sum_{i>0} p_i'(h) y^i\right]=\left[\displaystyle\sum_{i\geq 0} p_i(h) x^i+\displaystyle\sum_{i>0} p_i'(h) y^i\right]h.
\end{equation}
By comparing coefficients in Equation \ref{fsddf5367} we find that $p_i(h)=p_i'(h)=0$ for all odd $i$. Again by assumption we have $xu=ux$. This implies $p_i(h)=p_i'(qh^{-1})$ and $p_i'(h)=p_i'(qh^{-1})$ for all $i$. By induction we can show that a Laurent polynomial belongs to the ring $\Bbbk[h+qh^{-1}]$ if and only if $p(h)=p(qh^{-1})$. Thus, we find that $Z(\Bbbk[h^{\pm 1}](\sigma_{inv},a))$ is generated by $h+qh^{-1}$, $x^2$, and $y^2$. Since $x^2$ and $h+qh^{-1}$ are algebraically independent, we find that the GK dimension of $Z(\Bbbk[h^{\pm 1}](\sigma_{inv},a))$ greater than one. It is known that the GK dimension of a generalized Weyl algebra over $\Bbbk[h^{\pm 1}]$ is two and thus, the GK dimension of $Z(\Bbbk[h^{\pm 1}](\sigma_{inv},a))$ is two. It is clear that we can find a surjection say $f$ from $\Bbbk[h+h^{-1}](\mathrm{id},a(h)a(qh^{-1}))$ to $Z(\Bbbk[h^{\pm 1}](\sigma_{inv},a))$ by mapping generators to generators. Thus, $$\frac{\Bbbk[h+h^{-1}](\mathrm{id},a(h)a(qh^{-1}))}{\mathrm{ker}f}\cong Z(\Bbbk[h^{\pm 1}](\sigma_{inv},a)).$$ Since $\Bbbk[h+h^{-1}](\mathrm{id},a(h)a(qh^{-1}))$ is a domain of GK dimension two, if $f$ has a non-trivial kernel a contradiction would arise (see \cite[Proposition 3.15]{Lenagan78}). Hence we have the desired isomorphism. 
\end{proof}

Let $\Sigma$ and $\Sigma^{-1}$ be the multiplicative sets generated by $x$ and $y$ respectively. Note $\Sigma$ and $\Sigma^{-1}$ are both sets of regular elements in $\Bbbk[h^{\pm 1}](\sigma_{inv},a)$. For the statement and proof of Lemma \ref{loca989}, we use the notation $\sigma_{inv}'(h)=h$.
\begin{lemma}\label{loca989}
For $a(h)$ not a monomial, we have $$\Bbbk[h^{\pm 1}](\sigma_{inv},a)\Sigma\cong \Bbbk[h^{\pm 1}](\sigma_{inv}',1)~~\mathrm{and}~~\Bbbk[h^{\pm 1}](\sigma_{inv},a)\Sigma'\cong \Bbbk[h^{\pm 1}](\sigma_{inv}',1).$$ 
\end{lemma}
\begin{proof}
We only need show that $x$ and $y$ satisfy the Ore condition. Let $$r=\displaystyle\sum_{i\geq 0} p_i(h)x^i+\displaystyle\sum_{i=e}^d p_i'(h)y^i$$ where $p_i(h),p_i'(h)\in \Bbbk[h^{\pm{1}}]$ and $e,d\in\mathbb{Z}_{>0}$. It is clear that we can rewrite $r x^{d+1}=x r'$ where $r'\in \Bbbk[h^{\pm 1}](\sigma_{inv},a)$. The proof easily extends to any element in $\Sigma$ and showing that $y$ satisfies the Ore condition follows analogously. 
The result follows from Lemma \ref{iso1c} and the remark that follows it. 
\end{proof}

\begin{lemma}\label{fdkjh999990}
For $a$ not a monomial, every unit of $\Bbbk[h^{\pm 1}](\sigma_{inv},a)$ is of the form $\alpha h^i$ where $\alpha\in\Bbbk^*$ and $i\in\mathbb{Z}$.
\end{lemma}
\begin{proof}
Let $t$ be a unit of $\Bbbk[h^{\pm 1}](\sigma_{inv},a)$. Consider the two embeddings $$\phi:\Bbbk[h^{\pm 1}](\sigma_{inv},a)\hookrightarrow \Bbbk[u^{\pm 1}][v^{\pm 1},u\mapsto u^{-1}]$$ and $$\phi':\Bbbk[h^{\pm 1}](\sigma_{inv},a)\hookrightarrow \Bbbk[u^{\pm 1}][v^{\pm 1},u\mapsto u^{-1}]$$ defined in the following way:
$$\phi(x)=u,~~\phi(y)=a(v)u^{-1},~~\mathrm{and}~~\phi(h)=v$$   
and
$$\phi'(x)=u^{-1}a(v),~~\phi'(y)=u,~~\mathrm{and}~~\phi'(h)=v.$$  
From Corollary \ref{Units1265} and our embeddings we have that $t=\alpha h^i x^j$ and $t=\alpha' h^l y^k$ where $\alpha,\alpha'\in\Bbbk^*$ and $i,j,l,k\in\mathbb{Z}$. By comparing our expressions for $t$ and using the grading on $\Bbbk[h^{\pm 1}](\sigma_{inv},a)$ we get the desired result. 
\end{proof}

\subsection{Canonical Forms}

We now have all the properties required to prove that the families in Proposition \ref{prop_canonical_forms} are disjoint. For ease of reading, we repeat the statement of Proposition \ref{prop_canonical_forms} before working through the proof. 

\begin{proposition}\label{prop_canonical_forms4}
The generalized Weyl algebra $\Bbbk[h^{\pm 1}](\sigma, a)$ is isomorphic to one, and only one, of the following
\begin{enumerate}
\item $\Bbbk[h^{\pm 1}](\mathrm{id},a)$,
\item $\Bbbk[h^{\pm 1}](\sigma_q,a)$ where $\sigma_q(h)=qh$ and $q\neq 1$,
\item $\Bbbk[h^{\pm 1}](\sigma_{inv},a)$ where $\sigma_{inv}(h)=qh^{-1}$.
\end{enumerate}
\end{proposition}

\begin{proof}
Every algebra of the form $\Bbbk[h^{\pm 1}](\mathrm{id},a)$ is commutative, whereas families 2 and 3 only contain noncommutative algebras. Thus, family 1 is disjoint with family 2 and family 3.

Suppose $q$ is not a root of unity. Since $\Bbbk[h^{\pm 1}](\sigma_q,a)$ has a trivial center it cannot be isomorphic to any algebra in family 3. 

Let $q$ be a $t^{\mathrm{th}}$ root of unity for $t\in\mathbb{N}$ and note $h^t$ is central and a unit in $\Bbbk[h^{\pm 1}](\sigma_q,a)$. When $b$ is not a monomial, $\Bbbk[h^{\pm 1}](\sigma_{inv},b)$ has no nontrivial central units (see Lemma \ref{fsddf5367} and Lemma \ref{fdkjh999990}). Thus, these algebras cannot be isomorphic.   

For the remainder of the proof we set $g:=h, u:=x,$ and $v:=y$ in $\Bbbk[h^{\pm 1}](\sigma_q,a)$ to make clear which elements belong to which ring. We also assume $b$ is a monomial.

When $a(h)$ is not a monomial any isomorphism between $\Bbbk[h^{\pm 1}](\sigma_{inv},b)$ and $\Bbbk[h^{\pm 1}](\sigma_q,a)$, say $\psi$, would necessarily have action $\psi(h)=\alpha g^i$ and $\psi(x)=\beta g^j$ for $\alpha,\beta\in\Bbbk^*$. Applying $\psi$ to the relation $xh=qh^{-1}x$ gives us
$$\beta g^j \alpha g^i= q \alpha g^i \beta g^{-j}.$$
Comparing coefficients implies that $j=0$, contradicting the assumption that $\psi$ is an isomorphism.

Finally, assume $a(h)$ is a monomial and $\psi$ is an isomorphism between $\Bbbk[h^{\pm 1}](\sigma_{inv},b)$ and $\Bbbk[h^{\pm 1}](\sigma_q,a)$. By Corollary \ref{Units1265} we have that $\psi(h)=\alpha g^i u^j$ and $\psi(x)=\beta g^l u^k$ for $\alpha,\beta\in\Bbbk^*$ and $i,j,l,k\in\mathbb{Z}$.
Applying $\psi$ to the relation $xh=qh^{-1}x$ gives $$\beta g^l u^k \alpha g^i u^j= q\alpha^{-1}u^{-j} g^{-i}\beta g^l u^k$$
which implies that $i=j=0$ which contradicts $\psi$ being an isomorphism.   
\end{proof}

\subsection{Isomorphisms}\label{isogsgsrd}

In \cite{Bavula1} and \cite{Vivas} the isomorphism problem was solved for quantum generalized Weyl algebras over a Laurent polynomial ring. We now solve the isomorphism problem for $\Bbbk[h^{\pm 1}](\sigma_{inv},a)$ for $a\in\Bbbk[h^{\pm 1}]$, completing the classification of isomorphisms for the noncommative families in Proposition \ref{prop_canonical_forms}.

To identify in which algebra we are working, we attach the subscript $i\in \{1,2\}$ to the generators of each generalized Weyl algebra. Namely, we consider $\Bbbk[h^{\pm 1}](\sigma_{inv},a_1)$ and $ \Bbbk[h^{\pm 1}](\sigma_{inv}',a_2)$ generated by $h_1,x_1,y_1$ and $h_2,x_2,y_2$ respectively where $\sigma_{inv}(h_1)=q_1 h_1$ and $\sigma_{inv}'(h_2)=q_2 h_2$.  
\begin{proposition}\label{theisotw}
Let $a_1,a_2\in \Bbbk[h^{\pm 1}]$. Then the algebras $\Bbbk[h_1^{\pm 1}](\sigma_{inv},a_1)$ and $\Bbbk[h_2^{\pm 1}](\sigma_{inv}',a_2)$ are isomorphic, if and only if $$a_2(q_2^{\tau}h_2^{(-1)^{\tau}})=p h_2^l a_1(\alpha h_2^{\epsilon})$$ where $p\in\Bbbk^*, l\in\mathbb{Z}, \tau\in\{0,1\}, \epsilon \in \{1,-1\}$, and $\alpha^2=q_1q_2^{-\epsilon}$.
\end{proposition}
\begin{proof}
Suppose that $\psi$ is an isomorphism from $\Bbbk[h_1^{\pm 1}](\sigma_{inv},a_1)$ to $\Bbbk[h_2^{\pm 1}](\sigma_{inv}',a_2)$. 

For the case where both $a_1$ and $a_2$ are monomials see Lemma \ref{iso1c}.

If $a_1$ is a monomial and $a_2$ is not a monomial then a contradiction arises since $x_1^2$ is a central unit in $\Bbbk[h_1^{\pm 1}](\sigma_{inv},a_1)$, and no central units exist in $\Bbbk[h_2^{\pm 1}](\sigma_{inv},a_2)$ (see Section \ref{center_and_units}).    

Assume $a_1$ and $a_2$ are not monomials. Since $h_1$ is a unit, and $\psi$ is isomophism so has an inverse, we have that $\psi(h_1)=\alpha h_2^{\epsilon}$ where $\alpha\in\Bbbk^*$ and $\epsilon \in \{1,-1\}$. Applying $\Gamma$ to the relation $x_1h_1=q_1 h_1^{-1} x_1$ we get 
$$\psi(x_1)\alpha h_2^{\epsilon}=q_1\alpha^{-1}h_2^{-\epsilon}\psi(x_1)$$
which, under rearrangement gives 
$$q_1^{-1}\alpha^2 h^{\epsilon} \psi(x_1) h_2^{\epsilon}=\psi (x_1).$$ 
Using the $\mathbb{Z}$-graded structure described in Section \ref{grading1} we set $\psi(x_1)=\displaystyle\sum_{n\in\mathbb{Z}}W_n$ where $W_n\in A_{m}$ for $n\in \mathbb{Z}$. Thus, we have 
$$\displaystyle\sum_{m\in\mathbb{Z}}q_1^{-1} q_2^{\epsilon} \alpha^2  h_2^{\epsilon+(-1)^n \epsilon}  W_n=\displaystyle\sum_{n\in\mathbb{Z}}W_n$$
implying that $W_n=0$ for $n$ even and $\alpha^2=q_1 q_2^{-\epsilon}$. Applying an analogous argument to $\psi(y_1)$ and $y_1h_1=q_2h_1^{-1}y_1$ gives us that $$\psi(y_1)=\displaystyle\sum_{n\in\mathbb{Z}}W'_n$$ where $W'_n\in A_{m}$ for $n\in \mathbb{Z}$ and $W'_n=0$ for $n$ even. Applying $\psi$ to the relation $y_1 x_1=a_1(h_1)$ gives 
$$\displaystyle\sum_{n\in\mathbb{Z}}W'_n \displaystyle\sum_{n\in\mathbb{Z}}W_n = a_1(\alpha h_2^{\epsilon}).$$
Since $a_1(\alpha h_2^{\epsilon})\in \Bbbk[h_2^{\pm 1}]$, we find that $\psi(x_1)=W_n$ and $\psi(y_1)=W'_{-n}$ for some odd integer $n$. Thus, we can write $\psi(x_1)=p(h_2)x_2^{(1-\tau)n}y_2^{\tau n}$ and $\psi(y_1)=y_2^{(1-\tau)n}x_2^{\tau n}p'(h_2)$ where $p(h_2),p'(h_2)\in\Bbbk[h_2^{\pm 1}]$ and $\tau\in \{0,1\}$. By considering how $\psi^{-1}$ would act on our expressions for $\psi(x_1)$ and $\psi(y_1)$ we can deduce $\psi(x_1)=ph_2^i x_2^{(1-\tau)}y_2^{\tau}$ and $\psi(y_1)= y_2^{(1-\tau)}x_2^{\tau}p'h_2^j$ for $p,p'\in \Bbbk^*$ and $i,j\in\mathbb{Z}$. Returning to the application of $\psi$ to the relation $y_1x_1=a_1(h_1)$ we get
$$q^{i+j}pp'h_2^{-i-j}a_2(q_2^\tau h_2^{(-1)^{\tau}})=a_1(\alpha h_2^{\epsilon}).$$

Conversely, for $p\in\Bbbk^*, l\in\mathbb{Z}, \epsilon \in \{1,-1\},\alpha^2=q_1q_2^{-\epsilon}$, and assume $a_2( h_2^{(-1)^{\tau}})=p h_2^l a_1(\alpha h_2^{\epsilon})$. The map $\psi:\Bbbk[h_1^{\pm 1}](\sigma_{inv},a_1)\rightarrow \Bbbk[h_2^{\pm 1}](\sigma_{inv}',a_2)$ defined on the generators of $\Bbbk[h_1^{\pm 1}](\sigma_{inv},a_1)$ as follows:
$$\psi(h_1)=\alpha h_2^{\epsilon},~~\psi(x_1)=p^{-1} q_2^{-l} h_2^l x_2^{(1-\tau)} y_2^{\tau},~~\mathrm{and}~~\psi(y_1)=y_2^{(1-\tau)}x_2^{\tau}$$
is an isomorphism. In checking that $\psi$ is consistent on the defining relations of $\Bbbk[h_1^{\pm 1}](\sigma_{inv},a_1)$ it is necessary to derive the equation $a_2(q_2^{(-1)^{\tau}} h_2^{-(-1)^{\tau}})=q_2^l p h_2^{-l} a_1(\alpha q_2^{\epsilon} h_2^{-\epsilon})$ by applying $\sigma_{inv}'$ to the relation $a_2( h_2^{(-1)^{\tau}})=p h_2^l a_1(\alpha h_2^{\epsilon})$.  
\end{proof}

\subsection{Laurent Polynomial Identities}
We close our preliminaries section by presenting a number of Laurent polynomials identities that will be useful in the classification of the derivations of $\Bbbk[h^{\pm 1}](\sigma_{inv},a)$ in Section \ref{IGWA2}. We suggest this section is referenced as required whilst reading Section \ref{IGWA2}, since taken without that context, the identities proven may seem somewhat arbitrary. Throughout this section, we assume that our base field $\Bbbk$ does not have characteristic $2$.

\begin{lemma}\label{newlem33}
Let $p(h),g(h)\in\Bbbk[h^{\pm 1}]$ such that $p(h)+p(h^{-1})=g(h)(h^{-1}-h)$ then there exists $m(h)\in\Bbbk[h^{\pm 1}]$ such that $p(h)=m(h)(h^{-1}-h)$.
\end{lemma}
\begin{proof}
By assumption $1$ and $-1$ are roots of $p(h)+p(h^{-1})$ and thus, $1$ and $-1$ are roots of $p(h)$ given our assumption that the characteristic of $\Bbbk$ not equal to 2. 
\end{proof}

\begin{lemma}\label{lemma1.2poliden1}
Let $p(h)\in\Bbbk[h^{\pm 1}]$, then $p(h)=-h^2p(h^{-1})$ if and only if $p(h)=(h^2-1)g(h)$, where $g(h)=g(h^{-1})$. 
\end{lemma}
\begin{proof}
Assume $p(h)=(h^2-1)g(h)$, where $g(h)=g(h^{-1})$. Then 
\begin{align*}
p(h^{-1})=&(h^{-2}-1)g(h^{-1})\\
=&-h^{-2}(h^2-1)g(h)\\
=&-h^{-2}p(h)
\end{align*}
implying that $p(h)=-h^2p(h^{-1}).$

We now shift our attention to the sufficient condition. Assuming $p(h)=-h^2 p(h^{-1})$ and that the characteristic of $\Bbbk$ is not 2 gives us that $1$ and $-1$ are roots of $p(h)$. Therefore, we can write 
\begin{equation}\label{newlem43}
p(h)=g(h)(h^2-1). 
\end{equation}
Applying the automorphism of $\Bbbk[h^{\pm 1}]$ which sends $h\mapsto h^{-1}$ to Equation (\ref{newlem43}) gives us $$p(h^{-1})=g(h^{-1})(h^{-2}-1).$$ By assumption we can write $-h^{-2}p(h)=g(h^{-1})(h^{-2}-1).$ Combining this with Equation (\ref{newlem43}) implies that $g(h)=g(h^{-1}).$ 
\end{proof}

\begin{lemma}\label{lemsymanti}
Let $p(h)\in\Bbbk[h^{\pm 1}]$. There exists $g(h)\in\Bbbk[h^{\pm 1}]$ where $g(h)=-g(h^{-1})$, such that $p(h)-g(h)=p(h^{-1})-g(h^{-1}).$
\end{lemma}

\begin{proof}
Since $p(h)\in\Bbbk[h^{\pm 1}]$, we can write $p(h)=\displaystyle\sum_{i=1}^d p_ih^{i}+p_0+\displaystyle\sum_{i=1}^e p_i' h^{-i}$ where $p_i,p_i'\in\Bbbk$. Let $$g(h)=\frac{1}{2}\displaystyle\sum_{i=1}^d p_i(h^{i}-h^{-i})+\frac{1}{2}\displaystyle\sum_{i=1}^e p_i' (h^{-i}-h^{i}).$$
Direct calculation shows that $g(h)=-g(h^{-1}).$ We can see that 
\begin{align*}
p(h)-g(h)=&\frac{1}{2}\displaystyle\sum_{i=1}^d p_i(h^{i}+h^{-i})+p_0+\frac{1}{2}\displaystyle\sum_{i=1}^e p_i' (h^{i}+h^{-i})\\
&=p(h^{-1})-g(h^{-1}).
\end{align*}
\end{proof}
\begin{remark}
It is useful to note that for $p(h)\in\Bbbk[h^{\pm 1}]$, the condition $p(h)=-p(h^{-1})$ is equivalent to being able to write $p(h)=t(h)-t(h^{-1})$ for some $t(h)\in\Bbbk[h^{\pm 1}]$. 
\end{remark}

\section{Derivations of Involution Generalized Weyl Algebras}\label{IGWA2}
Throughout this section we assume that our base field $\Bbbk$ does not have characteristic $2$.

\subsection{Derivations of $\Bbbk[h^{\pm 1}](\sigma_{inv},1)$}

In this section we will determine the derivations of $\Bbbk[h^{\pm 1}](\sigma_{inv},1)$ where $\sigma_{inv}(h)=h^{-1}$. We will first determine the action of inner derivations.
\begin{lemma}\label{adDer1}
Let $t\in \Bbbk[h^{\pm 1}](\sigma_{inv},1)$, where $t=\displaystyle\sum_{k\geq 0} t_k(h)x^{k}+\displaystyle\sum_{k\geq 1} t_k'(h)y^{k}.$ Every inner derivation of $\Bbbk[h^{\pm 1}](\sigma_{inv},1)$, denoted $\mathrm{ad}_t$, has action
\begin{alignat*}{3}
\mathrm{ad}_t(h) &= \displaystyle\sum_{k\geq 0} (h^{-1}-h)t_{2k+1}(h)x^{2k+1}+\displaystyle\sum_{k\geq 1} (h^{-1}-h)t_{2k+1}'(h)y^{2k+1},\\
\mathrm{ad}_t(x) &= \displaystyle\sum_{k\geq0}\left[t_{k}(h)-t_{k}(h^{-1})\right]x^{k+1}+\displaystyle\sum_{k\geq 1} \left[t_{k}'(h)-t_{k}'(h^{-1})\right]y^{k-1},\\
\mathrm{ad}_t(y) &=
\displaystyle\sum_{k\geq0}\left[t_{k+1}(h)-t_{k+1}(h^{-1})\right]x^{k}+\displaystyle\sum_{k\geq1}\left[t_{k}'(h)-t_{k}'(h^{-1})\right]y^{k+1}+\left[t_{0}(h)-t_{0}(h^{-1})\right]y.
\end{alignat*}
\end{lemma}
\begin{proof}
We only show the calculation for $\mathrm{ad}_t(y)$ since $\mathrm{ad}_t(h)$ and $\mathrm{ad}_t(x)$ follow using the same reasoning. By definition $\mathrm{ad}_t(y)=ty-yt.$ Substituting for $t$ with the general element of $\Bbbk[h^{\pm 1}](\sigma_{inv},1)$, we get 
\begin{align*}
\mathrm{ad}_t(y)&=\left[\displaystyle\sum_{k\geq 0} t_k(h)x^{k}+\displaystyle\sum_{k\geq 1} t_k'(h)y^{k}\right] y -y \left[\displaystyle\sum_{k\geq 0} t_k(h)x^{k}+\displaystyle\sum_{k\geq 1} t_k'(h)y^{k}\right]\\
&=\displaystyle\sum_{k\geq 1} \left[t_k(h)-t_k(h^{-1})\right]x^{k-1}+\displaystyle\sum_{k\geq 1} \left[t_k'(h)-t_k'(h^{-1})\right]y^{k+1}  +[t_0(h)-t_0(h^{-1})]y.
\end{align*}
Reindexing, we get 
$$\mathrm{ad}_t(y) =
\displaystyle\sum_{k\geq0}\left[t_{k+1}(h)-t_{k+1}(h^{-1})\right]x^{k}+\displaystyle\sum_{k\geq1}\left[t_{k}'(h)-t_{k}'(h^{-1})\right]y^{k+1}+\left[t_{0}(h)-t_{0}(h^{-1})\right]y.$$
\end{proof}

\begin{lemma}\label{specialder21}
The map $\Dmath_{z_1, z_2}$ defined on the generators of $\Bbbk[h^{\pm 1}](\sigma_{inv},1)$ as follows
\begin{align*}
\Dmath_{z_1, z_2}(h) &= (h^2-1) z_1,\\
\Dmath_{z_1, z_2}(x) &= z_2 x,\\
\Dmath_{z_1, z_2}(y) &= -z_2 y,
\end{align*}
where $z_1, z_2\in Z(\Bbbk[h^{\pm 1}](\sigma_{inv},1))$, is an outer derivation of $\Bbbk[h^{\pm 1}](\sigma_{inv},1)$. 
\end{lemma}
\begin{proof}
For simplicity, we set $\textrm{D}:=\Dmath_{z_1, z_2}$ for the duration of this proof. First, we check that $\textrm{D}$ preserves the defining relations of $\Bbbk[h^{\pm 1}](\sigma_{inv},1)$.
Applying $\textrm{D}$ to $xh$, we see that 
\begin{align*}
\textrm{D}(xh) &= \textrm{D}(x)h+x\textrm{D}(h)\\
&= z_2xh+ x(h^2-1)z_1\\
&= h^{-1}z_2x + (h^{-2}-1)z_1 x\\
&= h^{-1}\textrm{D}(x) + h^{-1} \textrm{D}(x)\\
& = \textrm{D}(h^{-1}x).
\end{align*}
Applying $\textrm{D}$ to $yx$, we see that
\begin{align*}
\textrm{D}(yx) &= \textrm{D}(y)x+y\textrm{D}(x)\\
&=-z_2yx+yz_2x\\
&=0\\
&=\textrm{D}(1).
\end{align*}
The relations $yh=h^{-1}y$ and $xy=1$ follow similarly. 

To see that $\textrm{D}$ is an outer derivation we first note that $\mathrm{ad}_t(h)$ is the sum of elements from odd graded slices and all terms in $\mathrm{ad}_t(x)$ and $\mathrm{ad}_t(y)$ have coefficient polynomials of the form $p(h)=-p(h^{-1})$ (see Lemma \ref{adDer1}). Thus, we cannot represent $\textrm{D}$ as a sum of adjoint derivations.
\end{proof}

\begin{proposition}\label{Prop1cx}
Every derivation of $\Bbbk[h^{\pm 1}](\sigma_{inv},1)$ can be uniquely written as $\mathrm{ad}_t+\Dmath_{z_1,z_2}$ where $t\in\mathcal{T}_q$ and $\Dmath_{z_1,z_2}$ is defined as in Lemma \ref{specialder21}.
\end{proposition}

\begin{proof}
Let $\textrm{D}$ be a derivation of $\Bbbk[h^{\pm 1}](\sigma_{inv},1)$. By picking generic elements in $\Bbbk[h^{\pm 1}](\sigma_{inv},1)$ (as explained in Subsection \ref{grading1}), we get the the following action on the generators of $\Bbbk[h^{\pm 1}](\sigma_{inv},1)$: 
\begin{align*}
\textrm{D}(h) &= \displaystyle\sum_{k\geq 0}w_{k}(h)x^k + \displaystyle\sum_{k\geq 1}w_{k}'(h)y^k,\\
\textrm{D}(x) &= \displaystyle\sum_{k\geq 0}u_{k}(h)x^k + \displaystyle\sum_{k\geq 1}u_{k}'(h)y^k,\\
\textrm{D}(y) &= \displaystyle\sum_{k\geq 0}v_{k}(h)x^k + \displaystyle\sum_{k\geq 1}v_{k}'(h)y^k.
\end{align*}
Applying $\textrm{D}$ to the relation $yx=1$ gives us
\begin{equation}\label{fwafife2}
\textrm{D}(y)x+y\textrm{D}(x)=0. 
\end{equation}
Substituting for $\textrm{D}(y)$ and $\textrm{D}(x)$ in Equation (\ref{fwafife2}) we get 
\begin{equation}\label{eq31der}
\left[\displaystyle\sum_{k\geq 0}v_{k}(h)x^k + \displaystyle\sum_{k\geq 1}v_{k}'(h)y^k \right]x+y\left[\displaystyle\sum_{k\geq 0}u_{k}(h)x^k + \displaystyle\sum_{k\geq 1}u_{k}'(h)y^k \right] = 0.
\end{equation}
Rearranging Equation (\ref{eq31der}) we get
\begin{equation}\label{Eq134343}
\displaystyle\sum_{k\geq 0}v_{k}(h)x^{k+1} + \displaystyle\sum_{k\geq 1}v_{k}'(h)y^{k-1} +\displaystyle\sum_{k\geq 0}u_{k}(h^{-1})x^{k-1} + \displaystyle\sum_{k\geq 1}u_{k}'(h^{-1})y^{k+1}  = 0.
\end{equation}
By comparing coefficients in Equation (\ref{Eq134343}), we derive the four following relations:
\begin{align}\label{listrel1}
&v_1'(h)=-u_1(h^{-1}),\nonumber\\ 
&v_2'(h)= -u_{0}(h^{-1}),\nonumber\\
&v_{k-1}(h)=-u_{k+1}(h^{-1}),~\mathrm{for}~k\geq 1,\\
&v_{k+1}'(h)=-u_{k-1}'(h^{-1}),~\mathrm{for}~k> 1.\nonumber
\end{align}
Applying $\textrm{D}$ to the relation $xh=h^{-1}x$ we get
$$\textrm{D}(x)h+x\textrm{D}(h)=\textrm{D}(h^{-1})x+h^{-1}\textrm{D}(x).$$
Noting that $\textrm{D}(h^{-1})=-h^{-1}\textrm{D}(h)h^{-1}$ and substituting for $\textrm{D}(x)$ and $\textrm{D}(h)$ gives us

\begin{align}\label{Eq3423}
&\left[\displaystyle\sum_{k\geq 0}u_{k}(h)x^k + \displaystyle\sum_{k\geq 1}u_{k}'(h)y^k \right]h+x\left[\displaystyle\sum_{k\geq 0}w_{k}(h)x^k + \displaystyle\sum_{k\geq 1}w_{k}'(h)y^k\right]\\
&=-h^{-1}\left[\displaystyle\sum_{k\geq 0}w_{k}(h)x^k + \displaystyle\sum_{k\geq 1}w_{k}'(h)y^k\right]h^{-1}x+h^{-1}\left[\displaystyle\sum_{k\geq 0}u_{k}(h)x^k + \displaystyle\sum_{k\geq 1}u_{k}'(h)y^k \right].\nonumber
\end{align}
By comparing coefficients in Equation (\ref{Eq3423}) we derive the following five relations:
\begin{align}
&u_0(h)(h^{-1}-h)=w_1'(h)+w_1'(h^{-1}),\label{listofrel1}\\
&u_{2k}(h)(h^{-1}-h)= w_{2k-1}(h)+w_{2k-1}(h^{-1}),~\textrm{for}~ k\geq 1,\label{listofrel2}\\
&u_{2k}'(h)(h^{-1}-h) =  w_{2k+1}'(h)+w_{2k+1}'(h^{-1}),~\textrm{for}~ k\geq 1,\label{listofrel3}\\
&w_{2k}(h^{-1}) = -h^{-2}w_{2k}(h),~\textrm{for}~ k\geq 0,\label{listofrel4}\\
&w_{2k+2}'(h^{-1}) =- h^{-2}w_{2k+2}'(h),~\textrm{for}~ k\geq 0.\label{listofrel5}
\end{align}
By applying Lemma \ref{newlem33} to Equation (\ref{listofrel1}), Equation (\ref{listofrel2}), and Equation (\ref{listofrel3}) and applying Lemma \ref{lemma1.2poliden1} to Equation (\ref{listofrel4}) and Equation (\ref{listofrel5}) gives us    
\begin{align}\label{listrel2}
&u_0(h)=\frac{w_1'(h)+w_1'(h^{-1})}{(h^{-1}-h)},\nonumber\\
&u_{2k}(h)= \frac{w_{2k-1}(h)+w_{2k-1}(h^{-1})}{(h^{-1}-h)},~\textrm{for}~ k\geq 1,\nonumber\\
&u_{2k}'(h) = \frac{ w_{2k+1}'(h)+w_{2k+1}'(h^{-1})}{(h^{-1}-h)},~\textrm{for}~ k\geq 1,\\
&w_{2k}(h)=(h^{2}-1)z_{k}(h),~\textrm{for}~ k\geq 0,\nonumber\\
&w_{2k+2}'(h)=(h^{2}-1)z_{k}'(h),~\textrm{for}~ k\geq 0\nonumber
\end{align}
for $z_{k}(h),z_{k}'(h)\in Z(\Bbbk[h^{\pm 1}](\sigma_{inv},1))$.
Combining the relations displayed in Equation (\ref{listrel1}) and Equation (\ref{listrel2}), we find
\begin{align*}
\textrm{D}(h) &= (h^2-1)z_1 + \displaystyle\sum_{k\geq 0}w_{2k+1}(h)x^{2k+1} + \displaystyle\sum_{k\geq 0}w_{2k+1}'(h)y^{2k+1},\\
\textrm{D}(x) &= \displaystyle\sum_{k\geq 0} \frac{w_{2k+1}(h)+w_{2k+1}(h^{-1})}{(h^{-1}-h)}x^{2k+2}+ \displaystyle\sum_{k\geq 0} \frac{w_{2k+1}'(h)+w_{2k+1}'(h^{-1})}{(h^{-1}-h)}y^{2k}\\
&\hspace{4.3cm}+ \displaystyle\sum_{k\geq 0}  u_{2k+1}(h)x^{2k+1} +\displaystyle\sum_{k\geq 0}  u_{2k+1}'(h)y^{2k+1},\\
\textrm{D}(y) &= \displaystyle\sum_{k\geq 0} \frac{w_{2k+1}(h)+w_{2k+1}(h^{-1})}{(h^{-1}-h)}x^{2k}+ \displaystyle\sum_{k\geq 0} \frac{w_{2k+1}'(h)+w_{2k+1}'(h^{-1})}{(h^{-1}-h)}y^{2k+2}\\
&\hspace{1.75cm}-  \displaystyle\sum_{k\geq 1}  u_{2k+1}(h^{-1})x^{2k-1} -\displaystyle\sum_{k\geq 1} u_{2k-1}'(h^{-1})y^{2k+1}-u_1(h^{-1})y,
\end{align*}
where $z_1 \in Z(\Bbbk[h^{\pm 1}](\sigma_{inv},1)).$

We will now describe two derivations which will allow us to simplify the form of $\textrm{D}$. Consider the inner derivation $\mathrm{ad}_{t_1}$ where $t_1=\displaystyle\sum_{k\geq 0}\frac{w_{2k+1}(h)}{(h^{-1}-h)}x^{2k+1} + \displaystyle\sum_{k\geq 0}\frac{w_{2k+1}'(h)}{(h^{-1}-h)}y^{2k+1}$. Namely, $\mathrm{ad}_{t_1}$ has action

\begin{align*}
\mathrm{ad}_{t_1}(h) &= \displaystyle\sum_{k\geq 0}w_{2k+1}(h)x^{2k+1} + \displaystyle\sum_{k\geq 0}w_{2k+1}'(h)y^{2k+1},\\
\mathrm{ad}_{t_1}(x) &= \displaystyle\sum_{k\geq 0} \frac{w_{2k+1}(h)+w_{2k+1}(h^{-1})}{(h^{-1}-h)}x^{2k+2}+ \displaystyle\sum_{k\geq 0} \frac{w_{2k+1}'(h)+w_{2k+1}'(h^{-1})}{(h^{-1}-h)}y^{2k},\\
\mathrm{ad}_{t_1}(y) &=  \displaystyle\sum_{k\geq 0} \frac{w_{2k+1}(h)+w_{2k+1}(h^{-1})}{(h^{-1}-h)}x^{2k}+ \displaystyle\sum_{k\geq 0} \frac{w_{2k+1}'(h)+w_{2k+1}'(h^{-1})}{(h^{-1}-h)}y^{2k+2}.
\end{align*}
Consider the derivation $\Dmath_{z_1,0}$. Namely, $\Dmath_{z_1,0}$ has action
\begin{align*}
\Dmath_{z_1,0}(h) &= (h^2-1)z_1,\\
\Dmath_{z_1,0}(x) &= 0,\\
\Dmath_{z_1,0}(y) &=  0.
\end{align*}
We now compute the action of $\Dmath_1:=\Dmath-\mathrm{ad}_{t_{1}}-\Dmath_{z_1,0}$ on the generators of $\Bbbk[h^{\pm 1}](\sigma_{inv},1)$. We see that
\begin{align*}
\Dmath_1(h)&=0,\\
\Dmath_1(x)&= \displaystyle\sum_{k\geq 0}  u_{2k+1}(h)x^{2k+1}+\displaystyle\sum_{k\geq 0}  u_{2k+1}'(h)y^{2k+1},\\
\Dmath_1(y)&= -  \displaystyle\sum_{k\geq 1}  u_{2k+1}(h^{-1})x^{2k-1}-\displaystyle\sum_{k\geq 1}  u_{2k-1}'(h^{-1})y^{2k+1}-u_1(h^{-1})y.
\end{align*}
Applying Lemma \ref{lemsymanti} allows us to pick $t_2=\displaystyle\sum_{k\geq 0} t_{2k}(h)x^{2k}+\displaystyle\sum_{k\geq 1} t_{2k}'(h)y^{2k}$ such that
\begin{align*}
u_{2k+1}(h)-t_{2k}(h)+t_{2k}(h^{-1}) &= z_{2k+1}(h)\in Z(\Bbbk[h^{\pm 1}](\sigma_{inv},1)),\\
u'_{2k+1}(h)-t_{2k}'(h)+t_{2k}'(h^{-1}) &= z_{2k+1}'(h)\in Z(\Bbbk[h^{\pm 1}](\sigma_{inv},1))
\end{align*}
for $k\geq 0$. Thus, we can write $\Dmath_2=\Dmath_1-\mathrm{ad}_{t_{2}}$ in the following way:
\begin{align*}
\Dmath_2(h)&=0,\\
\Dmath_2(x)&= \displaystyle\sum_{k\geq 0}  z_{2k+1}(h)x^{2k+1}+\displaystyle\sum_{k\geq 0} z_{2k+1}'(h)y^{2k+1},\\
\Dmath_2(y)&= -  \displaystyle\sum_{k\geq 1}  z_{2k+1}(h)x^{2k-1}-\displaystyle\sum_{k\geq 0}  z_{2k+1}'(h)y^{2k+1}.
\end{align*}
Simplifying, we get
\begin{align*}
\Dmath_2(h)&=0,\\
\Dmath_2(x)&= z_2 x,\\
\Dmath_2(y)&= - z_2y,
\end{align*}
where $z_2\in Z(\Bbbk[h^{\pm 1}](\sigma_{inv},1))$.
Finally, we can express $\Dmath_2 = \Dmath_{0,z_2}$. Thus, we have shown that $\Dmath=\Dmath_{0,z_2}+\mathrm{ad}_{t_{2}}+\mathrm{ad}_{t_{1}}+\Dmath_{z_1,0}=\mathrm{ad}_{t_{1}+t_{2}}+\Dmath_{z_1,z_2}$ as required.
\end{proof}

\subsection{Derivations of $\Bbbk[h^{\pm 1}](\sigma_{inv},a)$}
In this section we will determine the derivations of $\Bbbk[h^{\pm 1}](\sigma_{inv},a)$ where $\sigma_{inv}(h)=qh^{-1}$ for $q\in\Bbbk^*$.

\begin{proposition}\label{profga3}
Every derivation of $\Bbbk[h^{\pm 1}](\sigma_{inv},a)$ can be uniquely written as $\mathrm{ad}_t+\Dmath_{z_1,z_2}$, where $t\in\mathcal{T}_q$ and $\Dmath_{z_1,z_2}$ are defined as in Lemma \ref{specialder21}.
\end{proposition}
\begin{proof}
By Lemma \ref{loca989} we have that every derivation $\Dmath$ of $\Bbbk[h^{\pm 1}](\sigma_{inv},a)$ uniquely extends to a derivation of $\Bbbk[h^{\pm 1}](\sigma_{inv},1)$. Thus, by Proposition \ref{Prop1cx} we have that $\Dmath=\mathrm{ad}_t+\Dmath_{z_1,z_2}$, where $t\in \Bbbk[h^{\pm 1}](\sigma_{inv},1)$ and $z_1,z_2\in Z(\Bbbk[h^{\pm 1}](\sigma_{inv},1)).$ Since the basis in which we have written the action of a derivation of $\Bbbk[h^{\pm 1}](\sigma_{inv},1)$ is compatible with the basis of $\Bbbk[h^{\pm 1}](\sigma_{inv},a)$, we can see that all the components of $\Dmath$, for example $\mathrm{ad}_t(h)+\Dmath_{z_1,z_2}(h),$ belong to $\Bbbk[h^{\pm 1}](\sigma_{inv},a)$ implying $t, z_1,$ and $z_2$ can all be chosen from $\Bbbk[h^{\pm 1}](\sigma_{inv},a)$. 

It remains to show that $\Dmath_{z_1,z_2}$ is consistent on the relations of $\Bbbk[h^{\pm 1}](\sigma_{inv},a)$. We first check $yx=a(h)$ and $xy=a(h^{-1})$. Thus, we consider 
$\textrm{D}(yx)=\textrm{D}(a(h)).$
By noting that 
\begin{align*}
\textrm{D}(yx)&=\textrm{D}(y)x+y\textrm{D}(x)\\
&=-z_2yx+y z_2 x\\
&=0
\end{align*}
we need only consider $\textrm{D}(a(h))=0$. Substituting for $a(h)$ we get
\begin{equation}\label{eqatsub3}
\textrm{D}(a(h))=\displaystyle\sum_{i=1}^d a_i \textrm{D}(h^i)+\textrm{D}(a_0)+\displaystyle\sum_{i=1}^e a_i' \textrm{D}(h^{-i}).
\end{equation}
Using the identities, which follow from general results pertaining to derivations (see \cite[Chapter III, Section 10]{Bourbaki}),
$$\textrm{D}(h^i)=\displaystyle\sum_{j=1}^i h^{j-1}\textrm{D}(h)h^{i-j}~~\textrm{and}~~\textrm{D}(h^{-i})=\displaystyle\sum_{j=1}^i h^{-j+1}\textrm{D}(h^{-1})h^{-i+j},$$
we rewrite the right hand side of Equation (\ref{eqatsub3}) as
\begin{equation}\label{actiona2}
\displaystyle\sum_{i=1}^d a_i \displaystyle\sum_{j=1}^i h^{j-1}\textrm{D}(h)h^{i-j} + \displaystyle\sum_{i=1}^e a_i' \displaystyle\sum_{j=1}^i h^{-j+1}\textrm{D}(h^{-1})h^{-i+j}.
\end{equation}
Using the identity
$\textrm{D}(h^{-1})=-h^{-1}\textrm{D}(h)h^{-1}$ and that $\textrm{D}=(h^2-1)z_1$ we get
\begin{equation}\label{actiona3}
\displaystyle\sum_{i=1}^d a_i \displaystyle\sum_{j=1}^i h^{i-1}(h^{2}-1)z_1 - \displaystyle\sum_{i=1}^e a_i' \displaystyle\sum_{j=1}^ih^{-i-1}(h^{2}-1)z_1.
\end{equation}
By simplifying and recalling that $\textrm{D}(a(h))=0$, we get 
$$(h^{2}-1)z_1
\left[ 
\displaystyle\sum_{i=1}^d i a_i h^{i-1} -
\displaystyle\sum_{i=1}^e i a_i' h^{-i-1}
\right]=0.$$
This implies that $z_1=0$. Using this, we can directly show that $\textrm{D}$ preserves $xh=h^{-1}x, yh=h^{-1}y$, and $xy=a(qh^{-1})$.
\end{proof}
\begin{remark}
Proposition \ref{Prop_main_2} follows as a corollary of Proposition \ref{Prop1cx} and Proposition \ref{profga3}. 
\end{remark}

\ \\ \hspace{1cm}
\begin{minipage}[c]{\linewidth}
~\\
\noindent 
Andrew P. Kitchin \\
{\tt andrew.p.kitchin@gmail.com} \\

\end{minipage}

\end{document}